\documentclass[12pt]{amsart}      
\usepackage{txfonts}
\usepackage{amssymb}
\usepackage{eucal}
\usepackage{graphicx}
\usepackage{amsmath}
\usepackage{amscd}
\usepackage[all]{xy}           
\usepackage[active]{srcltx} 
\usepackage{pifont}
\usepackage{amsfonts,latexsym}
\usepackage{xspace}
\usepackage{epsfig}
\usepackage{float}
\usepackage{color}
\usepackage{colordvi}
\usepackage{multicol}
\usepackage{colordvi}
\usepackage[hypertex]{hyperref}

\newtheorem{theorem}{Theorem}[section]
\newtheorem{prop}[theorem]{Proposition}
\newtheorem{defn}[theorem]{Definition}
\newtheorem{lemma}[theorem]{Lemma}
\newtheorem{coro}[theorem]{Corollary}
\newtheorem{prop-def}{Proposition-Definition}[section]
\newtheorem{coro-def}{Corollary-Definition}[section]

\newtheorem{exam}{Example}[theorem]

\newcommand{\nc}{\newcommand}

\nc{\tred}[1]{\textcolor{red}{#1}}
\nc{\tblue}[1]{\textcolor{blue}{#1}}
\nc{\tgreen}[1]{\textcolor{green}{#1}}
\nc{\tpurple}[1]{\textcolor{purple}{#1}}
\nc{\btred}[1]{\textcolor{red}{\bf #1}}
\nc{\btblue}[1]{\textcolor{blue}{\bf #1}}
\nc{\btgreen}[1]{\textcolor{green}{\bf #1}}
\nc{\btpurple}[1]{\textcolor{purple}{\bf #1}}

\renewcommand{\Bbb}{\mathbb}
\renewcommand{\frak}{\mathfrak}

\newcommand{\efootnote}[1]{}

\renewcommand{\textbf}[1]{}

\newcommand{\delete}[1]{}

\nc{\mlabel}[1]{\label{#1}}  
\nc{\mcite}[1]{\cite{#1}}  
\nc{\mref}[1]{\ref{#1}}  
\nc{\mbibitem}[1]{\bibitem{#1}} 

\delete{
\nc{\mlabel}[1]{\label{#1}  
{\hfill \hspace{1cm}{\bf{{\ }\hfill(#1)}}}}
\nc{\mcite}[1]{\cite{#1}{{\bf{{\ }(#1)}}}}  
\nc{\mref}[1]{\ref{#1}{{\bf{{\ }(#1)}}}}  
\nc{\mbibitem}[1]{\bibitem[\bf #1]{#1}} 
}

\nc{\mkeep}[1]{\marginpar{{\bf #1}}} 

\nc{\rb}{\mathrm{RB}}

\nc{\mapmonoid}{\frakM}
\nc{\disjoint}{\frakM'}
\nc{\ncpoly}[1]{\langle #1\rangle}  
\nc{\mapm}[1]{\{{#1}\}}             
\nc{\disj}[1]{\{{#1}\}'}
\nc{\mdisj}[1]{\frakM'(#1)}
\nc{\brho}{\bar{\rho}}
\nc{\om}{\bar{\frakm}}
\nc{\frakn}{\mathfrak n}
\nc{\ddeg}[1]{^{(#1)}}

\nc{\rbot}{\stackrel{RB}{\otimes}}
\nc{\CRB}{\mathfrak{CRB}}
\nc{\Comm}{\mathfrak{Comm}}
\nc{\bin}[2]{ (_{\stackrel{\scs{#1}}{\scs{#2}}})}  
\nc{\binc}[2]{ \left (\!\! \begin{array}{c} \scs{#1}\\
    \scs{#2} \end{array}\!\! \right )}  
\nc{\bincc}[2]{  \left ( {\scs{#1} \atop
    \vspace{-1cm}\scs{#2}} \right )}  
\nc{\bs}{\bar{S}}
\nc{\cosum}{\sqsubset}
\nc{\diff}[1]{\langle #1 \rangle}
\nc{\la}{\longrightarrow}
\nc{\rar}{\rightarrow}
\nc{\dar}{\downarrow}
\nc{\dprod}{**}
\nc{\dap}[1]{\downarrow \rlap{$\scriptstyle{#1}$}}
\nc{\md}{\mathrm{dth}}
\nc{\uap}[1]{\uparrow \rlap{$\scriptstyle{#1}$}}
\nc{\defeq}{\stackrel{\rm def}{=}}
\nc{\disp}[1]{\displaystyle{#1}}
\nc{\dotcup}{\ \displaystyle{\bigcup^\bullet}\ }
\nc{\gzeta}{\bar{\zeta}}
\nc{\hcm}{\ \hat{,}\ }
\nc{\hts}{\hat{\otimes}}
\nc{\barot}{{\otimes}}
\nc{\free}[1]{\bar{#1}}
\nc{\uni}[1]{\tilde{#1}}
\nc{\hcirc}{\hat{\circ}}
\nc{\leng}{\ell}
\nc{\lleft}{[}
\nc{\lright}{]}
\nc{\lc}{\lfloor}
\nc{\rc}{\rfloor}
\nc{\curlyl}{\left \{ \begin{array}{c} {} \\ {} \end{array}
    \right .  \!\!\!\!\!\!\!}
\nc{\curlyr}{ \!\!\!\!\!\!\!
    \left . \begin{array}{c} {} \\ {} \end{array}
    \right \} }
\nc{\longmid}{\left | \begin{array}{c} {} \\ {} \end{array}
    \right . \!\!\!\!\!\!\!}
\nc{\onetree}{\bullet}
\nc{\ora}[1]{\stackrel{#1}{\rar}}
\nc{\ola}[1]{\stackrel{#1}{\la}}
\nc{\ot}{\otimes}
\nc{\mot}{{{\boxtimes\,}}}
\nc{\otm}{\overline{\boxtimes}}
\nc{\sprod}{\bullet}
\nc{\scs}[1]{\scriptstyle{#1}}
\nc{\mrm}[1]{{\rm #1}}
\nc{\msum}{\sum\limits}
\nc{\margin}[1]{\marginpar{\rm #1}}   
\nc{\dirlim}{\displaystyle{\lim_{\longrightarrow}}\,}
\nc{\invlim}{\displaystyle{\lim_{\longleftarrow}}\,}
\nc{\mvp}{\vspace{0.3cm}}
\nc{\tk}{^{(k)}}
\nc{\tp}{^\prime}
\nc{\ttp}{^{\prime\prime}}
\nc{\svp}{\vspace{2cm}}
\nc{\vp}{\vspace{8cm}}
\nc{\proofbegin}{\noindent{\bf Proof: }}
\nc{\proofend}{$\blacksquare$ \vspace{0.3cm}}
\nc{\modg}[1]{\!<\!\!{#1}\!\!>}
\nc{\intg}[1]{F_C(#1)}
\nc{\lmodg}{\!<\!\!}
\nc{\rmodg}{\!\!>\!}
\nc{\cpi}{\widehat{\Pi}}
\nc{\sha}{{\mbox{\cyr X}}}  
\nc{\ssha}{{\mbox{\cyrs X}}}
\nc{\shap}{{\mbox{\cyrs X}}} 
\nc{\shpr}{\diamond}    
\nc{\shp}{\ast}
\nc{\shplus}{\shpr^+}
\nc{\shprc}{\shpr_c}    
\nc{\msh}{\ast}
\nc{\zprod}{m_0}
\nc{\oprod}{m_1}
\nc{\vep}{\varepsilon}
\nc{\labs}{\mid\!}
\nc{\rabs}{\!\mid}

\nc{\dth}{d}
\nc{\mmbox}[1]{\mbox{\ #1\ }}
\nc{\fp}{\mrm{FP}} \nc{\rchar}{\mrm{char}} \nc{\Fil}{\mrm{Fil}}
\nc{\Mor}{Mor\xspace}
\nc{\gmzvs}{gMZV\xspace}
\nc{\gmzv}{gMZV\xspace}
\nc{\mzv}{MZV\xspace}
\nc{\mzvs}{MZVs\xspace}
\nc{\Hom}{\mrm{Hom}} \nc{\id}{\mrm{id}} \nc{\im}{\mrm{im}}
\nc{\incl}{\mrm{incl}} \nc{\map}{\mrm{Map}} \nc{\mchar}{\rm char}
\nc{\nz}{\rm NZ} \nc{\supp}{\mathrm Supp}

\nc{\Alg}{\mathbf{Alg}}
\nc{\Bax}{\mathbf{Bax}}
\nc{\bff}{\mathbf f}
\nc{\bfk}{{\bf k}}
\nc{\bfone}{{\bf 1}}
\nc{\bfx}{\mathbf x}
\nc{\bfy}{\mathbf y}
\nc{\base}[1]{\bfone^{\otimes ({#1}+1)}} 
\nc{\Cat}{\mathbf{Cat}}

\nc{\detail}{\marginpar{\bf More detail}
    \noindent{\bf Need more detail!}
    \svp}
\nc{\Int}{\mathbf{Int}}
\nc{\Mon}{\mathbf{Mon}}
\nc{\rbtm}{{shuffle }}
\nc{\rbto}{{Rota-Baxter }}
\nc{\remarks}{\noindent{\bf Remarks: }}
\nc{\Rings}{\mathbf{Rings}}
\nc{\Sets}{\mathbf{Sets}}

\nc{\BA}{{\Bbb A}} \nc{\CC}{{\Bbb C}} \nc{\DD}{{\Bbb D}}
\nc{\EE}{{\Bbb E}} \nc{\FF}{{\Bbb F}} \nc{\GG}{{\Bbb G}}
\nc{\HH}{{\Bbb H}} \nc{\LL}{{\Bbb L}} \nc{\NN}{{\Bbb N}}
\nc{\KK}{{\Bbb K}} \nc{\QQ}{{\Bbb Q}} \nc{\RR}{{\Bbb R}}
\nc{\TT}{{\Bbb T}} \nc{\VV}{{\Bbb V}} \nc{\ZZ}{{\Bbb Z}}

\nc{\cala}{{\mathcal A}} \nc{\calc}{{\mathcal C}}
\nc{\cald}{{\mathcal D}} \nc{\cale}{{\mathcal E}}
\nc{\calf}{{\mathcal F}} \nc{\calg}{{\mathcal G}}
\nc{\calh}{{\mathcal H}} \nc{\cali}{{\mathcal I}}
\nc{\call}{{\mathcal L}} \nc{\calm}{{\mathcal M}}
\nc{\caln}{{\mathcal N}} \nc{\calo}{{\mathcal O}}
\nc{\calp}{{\mathcal P}} \nc{\calr}{{\mathcal R}}
\nc{\cals}{{\mathcal S}}
\nc{\calt}{{\mathcal T}} \nc{\calw}{{\mathcal W}}
\nc{\calk}{{\mathcal K}} \nc{\calx}{{\mathcal X}}
\nc{\CA}{\mathcal{A}}

\nc{\fraka}{{\frak a}}
\nc{\frakA}{{\frak A}}
\nc{\frakb}{{\frak b}}
\nc{\frakB}{{\frak B}}
\nc{\frakH}{{\frak H}}
\nc{\frakM}{{\frak M}}
\nc{\bfrakM}{\overline{\frakM}}
\nc{\frakm}{{\frak m}}
\nc{\frakP}{{\frak P}}
\nc{\frakN}{{\mathfrak N}}
\nc{\frakp}{{\frak p}}
\nc{\frakr}{{\frak r}}
\nc{\frakS}{{\frak S}}
\nc{\frakx}{{\frak x}}
\nc{\ox}{\bar{\frakx}}
\nc{\frakX}{{\mathfrak X}}
\nc{\fraky}{{\frak y}}

\font\cyr=wncyr10
\font\cyrs=wncyr7
\definecolor{chu}{rgb}{0,0.5,0}
\definecolor{guo}{rgb}{0.8,0,0}

\nc{\redt}[1]{\textcolor{red}{#1}}
\nc{\li}[1]{\textcolor{red}{Li:#1}}
\nc{\ch}[1]{\textcolor{blue}{Chenghao:#1}}

\begin{document}

\title{Localization of Rota-Baxter algebras}
%
%
\author{Chenghao Chu}
\address{Department of Mathematics, Johns Hopkins University, Baltimore, MD 21218}
\email{cchu@math.jhu.edu}
\author{Li Guo}
\address{Department of Mathematics and Computer Science,
         Rutgers University,
         Newark, NJ 07102}
\email{liguo@rutgers.edu}


\date{\today}

\begin{abstract}
A commutative Rota-Baxter algebra can be regarded as a commutative algebra that carries an abstraction of the integral operator. With the motivation of generalizing the study of algebraic geometry to Rota-Baxter algebra, we extend the central concept of localization for commutative algebras to commutative Rota-Baxter algebras. The existence of such a localization is proved and, under mild conditions, its explicit constructions are obtained. The existence of tensor products of commutative Rota-Baxter algebras is also proved and the compatibility of localization and tensor product of Rota-Baxter algebras is established. We further study Rota-Baxter coverings and show that they form a Gr\"othendieck topology.
\end{abstract}

\maketitle

\tableofcontents

\setcounter{section}{0}


\section{Introduction}
\mlabel{sec:int}
Throughout this paper, all algebras are assumed to be commutative over a commutative unitary ring $\bfk$.
A well-known concept in mathematics is that of a differential algebra, defined to be a $\bfk$-algebra $R$ with a $\bfk$-linear operator $d$ on $R$ that satisfies the Leibnitz rule:
$$ d(xy)=d(x)y+xd(y), \quad \forall x,y\in R.$$
Differential algebras originated from the algebraic study of differential equations by F. Ritt and E. Kolchin~\mcite{Ri,Kol} in the last century. It is a natural yet profound extension of commutative algebra and the related algebraic geometry. It has also found important applications in arithmetic geometry, logic and computational algebra, especially in the well-known work of W.~T. Wu~\mcite{Wu} on mechanical theorem proving in geometry. The theory of algebraic geometry was formulated by Kolchin in the language of Weil in the last century. The corresponding theory in the language of Grothendieck is being studied intensively in recent years~\mcite{Gi,GKOS,SP}.
\smallskip

As an integral analogue of a differential algebra,
a Rota-Baxter algebra (of weight zero) is a $\bfk$-algebra $R$ with a $\bfk$-linear operator $P$ on $R$ that satisfies the following abstraction of the integration by parts formula:
$$ P(x)P(y)=P(xP(y))+P(P(x)y), \quad \forall x,y\in R.$$
In the special case when $R$ is taken to be the algebra of continuous functions on ${\mathbb R}$ and $P$ is the integral operator $P(f)(t)=\int_0^t f(s)\,ds$ for $f(t)\in R$, the above formula is the integration by parts formula in calculus.
See Definition~\mref{de:back} for the definition of a Rota-Baxter algebra in general.

Rota-Baxter algebra started with the probability study of G. Baxter in 1960 and has since found applications in many areas of mathematics and physics, such as combinatorics (quasi-symmetric functions), number theory (multiple zeta values), operads (dendriform algebras), Yang-Baxter equations (after the well-known physicists C. Yang and R. Baxter), especially the profound work of Connes and Kreimer on renormalization of quantum field theory~\mcite{Ag,Bai,C-K0,E-G-K3,EGM,Guw,GPXZ,G-Z}. Nevertheless, systematic theoretic study of Rota-Baxter algebras was carried out only recently. After the work of Cartier and Rota in 1970s, free commutative Rota-Baxter algebras were constructed in terms of variations of shuffles in~\mcite{Gudom,G-K1,G-K2}. These algebras are the analogue of polynomial algebras in commutative algebra or differential polynomial algebras in differential algebra. However the algebraic structure of free commutative Rota-Baxter algebras is much more involved. For example, it is a simple and basic fact that free commutative algebras and free differential commutative algebras are polynomial algebras. Such a statement for a free Rota-Baxter algebra is either non-trivial to prove (in the zero characteristic case by using Lyndon words) or simply incorrect (in the positive characteristic case).
Indeed, establishing a theory of commutative algebra and algebraic geometry in a sense comparable to differential algebra, not to mention commutative algebra, is a task that has not even been started.
\smallskip

As a step in this direction, we need to develop a suitable localization theory for commutative Rota-Baxter algebras. Here, the difficulty in comparison to differential algebras is already evident from a primitive point of view: while the derivation for a quotient can be easily derived from the derivations of the numerator and denominator by the quotient rule, there is no general way that one can derive the integral of a quotient from its numerator and denominator.

Let $S$ be a subset of a commutative Rota-Baxter algebra $R$. We first prove in Section~\mref{sec:local} the existence of the localization of $R$ by $S$ in the category of commutative Rota-Baxter algebras,  called the Rota-Baxter localization at $S$. Then, under a mild restriction of a suitable decomposition of $S^{-1}R$, the usual localization of a commutative algebra, we gave an explicit construction of {Rota-Baxter localizations} in Section~\mref{sec:const}. We also note that even if the commutative Rota-Baxter algebra has zero Rota-Baxter operator, its Rota-Baxter localization is very different from the usual localization. In Section~\mref{sec:site}, we construct the tensor product of two commutative Rota-Baxter algebras. We further define a collection of { Rota-Baxter coverings} and show that it forms a Grothendieck topology. 

\noindent
{\bf Acknowledgements.} The authors thank Zongzhu Lin for helpful discussions. Li Guo thanks NSF grant DMS-1001855 for support.

\section{Localization}
\mlabel{sec:local}
In this section, we first put together background on Rota-Baxter algebras that we  need. We then define the localization of a commutative Rota-Baxter algebra and prove its existence by concrete constructions.

\subsection{Background on Rota-Baxter algebras}
\mlabel{ss:back}
Let $\bfk$ be a commutative unitary ring with unit $1_\bfk\in \bfk$.

\begin{defn}
{\rm
\begin{enumerate}
\item Let $\lambda\in \bfk$ be given.
A {\bf Rota-Baxter $\bfk$-algebra of weight $\lambda$} is a pair $(R, P)$ where $R$ is a $\bfk$-algebra and $P\colon\, R\longrightarrow R$ is a $\bfk$-linear map, called a {\bf Rota-Baxter operator}, satisfying
\begin{equation}
P(x)P(y)=P(xP(y)+P(x)y+\lambda xy)\mbox{ for all } x,y\in R.
\mlabel{eq:rbe}
\end{equation}
\item
A {\bf morphism} $f\colon\, (R, P)\to (S, Q)$ of Rota-Baxter algebras is a $\bfk$-algebra homomorphism $f:R\to S$ such that
$$f(P(a))=Q(f(a)) \mbox{ for all } a\in R.$$
\item
A subalgebra (resp. An ideal) of $(R,P)$ is called a {\bf Rota-Baxter subalgebra} (resp. {\bf Rota-Baxter ideal}) of $(R,P)$ if it is closed under $P$.
\item
Let $I$ be a Rota-Baxter ideal of $(R, P)$. We let $(R/I,\overline{P})$ denote the Rota-Baxter algebra with the Rota-Baxter operator $\overline{P}$ defined by $\overline{P}(\overline{a})=\overline{P(a)} \mbox{ for all } \overline{a}\in R/I.$
\item
Let $(R,P_R)$ be a Rota-Baxter algebra and let $S\subseteq R$ be a subset. A Rota-Baxter subalgebra $(B, P_R)$ of $(R,P_R)$ is said to be generated by $S$ if it is the smallest Rota-Baxter algebra containing $S$, or equivalently, the intersection of all Rota-Baxter subalgebras containing $S$. \item
Let $R'$ be a subalgebra of $R$ which may not be closed under $P_R$, then $(R,P_R)$ is said to be finitely generated over $R'$ as a Rota-Baxter algebra if it is generated by the union of $R'$ and some finite subset $S\subseteq R$.
\end{enumerate}
}
\mlabel{de:back}
\end{defn}

Let $\CRB/\bfk$ denote the category of commutative Rota-Baxter $\bfk$-algebras and let $\Comm/\bfk$ denote the usual category of commutative $\bfk$-algebras. Clearly, we have the forgetful functor $$\mathrm{F}: \CRB/\bfk\longrightarrow \Comm/\bfk$$
by forgetting the Rota-Baxter operators. In \mcite{G-K1}, a free commutative Rota-Baxter algebra of weight $\lambda$ on a commutative $\bfk$-algebra $A$ is constructed in terms of a generalization of the shuffle product, called the mixable shuffle product which is a natural generalization of the quasi-shuffle product~\mcite{Ho}. This free commutative Rota-Baxter algebra on $A$ is denoted by $\sha (A)$. As a $\bfk$-module, we have
$$\sha (A) = \bigoplus\limits_{i\geq 1}A^{\otimes i}=A\oplus (A\otimes A)\oplus (A\otimes A\otimes A)\oplus\cdots$$ where the tensor is defined over $\bfk$. The multiplication on $\sha A$ is taken to be the product $\diamond$ defined as follows. Let $\mathfrak{a}=a_0\otimes \cdots \otimes a_m\in A^{\otimes(m+1)}$ and $\mathfrak{b}=b_0\otimes \cdots \otimes b_n\in A^{\otimes(n+1)}$. If $mn=0$, define

\begin{equation}
\mathfrak{a}\diamond \mathfrak{b}=
\begin{cases}
(a_0b_0)\otimes b_1 \otimes\cdots \otimes b_n, & m=0, n>0, \\
(a_0b_0)\otimes a_1 \otimes\cdots \otimes a_m, & m>0, n=0, \\
a_0b_0,                                        & m=n=0. \\
\end{cases}
\mlabel{eq:mshprod1}
\end{equation}
If $m>0$ and $n>0$, then $\mathfrak{a}\diamond \mathfrak{b}$ is inductively, on $m$ and $n$,  defined by
\begin{equation}
(a_0b_0)\otimes\big((a_1\otimes\cdots\otimes a_m)\diamond (1\otimes b_1\otimes \cdots b_n)
                 +(1\otimes a_1\otimes \cdots\otimes a_m)\diamond (b_1\otimes \cdots b_n)
                 +\lambda (a_1\otimes\cdots\otimes a_m)\diamond (b_1\otimes \cdots b_n)\big).
\mlabel{eq:mshprod2}
\end{equation}
The Rota-Baxter operator $P_{\sha A}$ on $\sha A$ is defined by
\begin{equation}
P_{\sha A}(x_0\otimes \cdots \otimes x_n)=1_A \otimes x_0\otimes\cdots\otimes x_n.
\mlabel{eq:rbo}
\end{equation}
It is proved in \cite[Corollary 4.3]{G-K1} that $(\sha, \mathrm{F})$ is an adjoint pair.

We also display the following statement for later references.
\begin{prop} The free Rota-Baxter algebra $(\sha A, P_{\sha A})$ on $A$ is generated by $A$ as a Rota-Baxter algebra.
\mlabel{prop:finteness}
\end{prop}
This follows from general principles of free objects in universal algebra. More precisely, let $(F,P)$ be the commutative Rota-Baxter subalgebra of $(\sha A,P_{\sha A})$ generated by $A$. Then, by taking restriction, the universal property of $(\sha A,P_{\sha A})$ gives the universal property of $(F,P)$.

\subsection{The existence of localization}
\mlabel{ss:exist}

We now define the localization of a commutative Rota-Baxter algebra.
\begin{defn}
Let $(R,P)$ be a commutative Rota-Baxter algebra. Let $S\subseteq R$ be a multiplicative subset. The {\bf Rota-Baxter localization of $R$ at $S$} is a commutative Rota-Baxter algebra, denoted by  $(S^{-1}_\rb R, S^{-1}P)$, together with a Rota-Baxter algebra homomorphism
$i_S: R\to S^{-1}_\rb R$ satisfying the following properties:
\begin{enumerate}
\item
The elements in $i_S(S)\subseteq S^{-1}_\rb R$ are invertible;
\item
For any commutative Rota-Baxter algebra $(R',P_{R'})$ and Rota-Baxter algebra homomorphism $f: R\to R'$ such that $f(S)\subseteq R'$ is a set of invertible elements, there is a unique $f_S: S^{-1}_\rb R\to R'$ such that $f_S\circ i_S = f$.
\end{enumerate}
\mlabel{de:local}
\end{defn}


It follows from the definition that the Rota-Baxter localization of $R$ at a multiplicative subset $S$ is unique up to isomorphisms, if it exists.  We next prove the existence of Rota-Baxter localization. Let $S^{-1}R$ be the localization of the commutative algebra $R$ at $S$. Let $(\sha(S^{-1}R),P_{\sha(S^{-1}R)})$ be the free commutative Rota-Baxter algebra on $S^{-1}R$ as constructed in~\mcite{G-K1} and recalled in Section~\mref{ss:back}. So we have
$$ \sha(S^{-1}R)=\bigoplus_{k\geq 1} (S^{-1}R)^{\ot k}, \quad
P_{\sha(S^{-1}R)}(a_1\ot \cdots \ot a_k)=1\ot a_1\ot \cdots \ot a_k, \forall a_1\ot \cdots \ot a_k \in (S^{-1}R)^{\ot k}.$$
To simplify notations, we also write $a_1\ot \cdots \ot a_k\in R^{\ot k}$ for its image in
$(S^{-1}R)^{\ot k}$.

\begin{theorem}
Let $(R,P_R)$ be a commutative Rota-Baxter algebra and $S\subseteq R$ be a multiplicative subset.
Let $I_{S^{-1}R}$ be the Rota-Baxter ideal of $\sha(S^{-1}R)$ generated by the set
\begin{equation}
\{ P_R(a) - 1\ot a\ |\ a\in R \}.
\mlabel{eq:localid}
\end{equation}
Let $(\sha(S^{-1}R)/I_{S^{-1}R}, \overline{P_{\sha(S^{-1}R)}})$ be the corresponding quotient Rota-Baxter algebra. Let
\begin{equation}
i: (R, P_R)\to (\sha(S^{-1}R)/I_{S^{-1}R}, \overline{P_{\sha(S^{-1}R)}}),\,\, a\mapsto a, \,\, a\in R.
\mlabel{eq:iloc}
\end{equation}
Then the triple $(\sha(S^{-1}R)/I_{S^{-1}R}, \overline{P_{\sha(S^{-1}R)}},i)$ is the localization of $(R,P_R)$ at $S$.
\mlabel{thm:exist}
\end{theorem}
Because of the theorem, we will use $(\sha(S^{-1}R)/I_{S^{-1}R}, \overline{P_{\sha(S^{-1}R)}},i)$ to denote the localization $(S^{-1}_{RB}R,S^{-1}P,j_S)$.

\begin{proof}
We just need to verify that the triple $(\sha(S^{-1}R)/I_{S^{-1}R}, \overline{P_{\sha(S^{-1}R)}},i)$ satisfies the universal property of Rota-Baxter algebra localization.

Assume that $f:  (R, P_R)\to (R', P_{R'})$ is a  morphism of commutative Rota-Baxter algebras over $\bfk$. If $f(S)$ is invertible in $R'$, we have a $\bfk$-algebra morphism $S^{-1}f: S^{-1}R \to R'$. This induces a morphism of commutative Rota-Baxter algebras $\sha S^{-1}f: (\sha(S^{-1}R),P_{\sha(S^{-1}R)}) \to (R', P_{R'})$. Let $P_R(a)-1\otimes a$ be a generator of the Rota-Baxter ideal $I_{S^{-1}R}$, then
\begin{center}\begin{tabular}{cl}& $\sha S^{-1}f(P_R(a)-1\otimes a)$\\
 = & $\sha S^{-1}f(P_R(a))-\sha S^{-1}f(1\otimes a)=f(P_R(a))-\sha S^{-1}f(P_{\sha(S^{-1}R)}(a) )$\\
 = & $P_{R'}(f(a))-P_{R'}(f(a) )=0.$
\end{tabular}\end{center}
So $I_{S^{-1}R}$ is in the kernel of $f$ and hence we have an induced morphism $$S^{-1}_{RB}f: (S^{-1}_{RB}R, S^{-1}P_R) \longrightarrow (R', P_R')$$ which satisfies $f=S^{-1}_{RB}f\circ i$. If $g: (S^{-1}_{RB}R, S^{-1}P_R) \to (R', P_{R'})$ is a Rota-Baxter algebra morphism which also satisfies $f=g\circ i$, then we have a morphism
$$G =g\circ \pi:  (\sha(S^{-1}R),P_{\sha(S^{-1}R)})\longrightarrow (R', P_R') $$  where $\pi$ is the obvious quotient map. Let $j: R\to \sha(S^{-1}R)$ be the obvious algebra morphism. It is easy to see that $G\circ j=f$ and hence $G|_{S^{-1}R}=S^{-1}f$. By the universal property of free commutative Rota-Baxter algebras, this further implies that $G= \sha S^{-1}f$, i.e., $g\circ \pi = S^{-1}_{RB}f\circ \pi$. Since $\pi$ is surjective, we have $g=S^{-1}_{RB}f.$ The proof is completed.
\end{proof}
\begin{coro} Let $(R, P_R)$ be a commutative Rota-Baxter algebra and $S$ be a finitely generated multiplicative subset of $R$. Then the localization $(S^{-1}_{RB}R, S^{-1}P)$ is finitely generated over $R$ as an Rota-Baxter algebra.
\end{coro}
\begin{proof} It follows  from the proof of Theorem~\mref{thm:exist} that  $(S^{-1}_{RB}R, S^{-1}P)$ is a quotient of the free Rota-Baxter algebra  $(\sha S^{-1}R, P_{\sha S^{-1}}R)$. The latter is generated by $S^{-1}R$ as a Rota-Baxter algebra by Proposition \ref{prop:finteness}. Since $S$ is a finitely generated multiplicative subset, $S^{-1}R$ is a finitely generated algebra over $R$. Thus $(S^{-1}_{RB}R, S^{-1}P)$ is finitely generated over $R$ as a Rota-Baxter algebra.
\end{proof}

\section{Constructions of localization}
\mlabel{sec:const}
For further study of algebraic geometry of Rota-Baxter algebras, we need to give explicit constructions of Rota-Baxter algebra localization. We assume that $S^{-1}R=R \oplus V$ as $\bfk$-modules, where the direct summand $R$ denotes its image in $S^{-1}R$ in order to simply notations. This is true for example when $\bfk$ is a field. We also compare the Rota-Baxter algebra localization with the usual localization of commutative algebras.

\subsection{The general weight case}

\begin{theorem} Let $(R, P_R)$ be a commutative Rota-Baxter algebra of weight $\lambda$ over $\bfk$. Assume that $S$ is a multiplicative subset of $R$ such that $S^{-1}R=R \oplus V$ for a nonunitary subring $V$ of $R$.
Then the Rota-Baxter localization of $R$ at $S$ is given by
$$S^{-1}_\rb R= S^{-1}R\oplus \bigoplus_{k\geq 1} (S^{-1}R \ot V^{\ot k})$$
with the multiplication given by the mixable shuffle product.
\mlabel{thm:locnzero}
\end{theorem}

The equation in the theorem can be alternatively stated as the tensor product algebra
$$ S^{-1}_\rb R = S^{-1}R \ot \Big( \bigoplus_{k\geq 0}  V^{\otimes k}\Big)
$$
where the product in the second tensor factor is given by the mixable shuffle product or quasi-shuffle product~\mcite{EGsh,G-Z2}.

We also note that the right hand side of the equation in the theorem can be regarded as a subalgebra of $\sha(S^{-1}R)$ since we require that $V$ is closed under the multiplication in $S^{-1}R$ and hence in $\sha (S^{-1}R)$. However, this subalgebra is not a Rota-Baxter subalgebra of $(\sha(S^{-1}R), P_{\sha(S^{-1}R)})$ as we will see in the proof that the Rota-Baxter operators are different.

\begin{proof} Let $B$ denote the subalgebra $S^{-1}R\oplus \bigoplus_{k\geq 1} (S^{-1}R \ot V^{\ot k})$ of $\sha(S^{-1}R)$.
We have a natural map of algebras
\begin{equation}
f: B \longrightarrow \sha S^{-1}R \longrightarrow S^{-1}_\rb R.
\mlabel{eq:locmap}
\end{equation}
We will define a Rota-Baxter operator $P:=P_B$ of weight $\lambda$ on $B$ such that $f$ is an isomorphism of Rota-Baxter algebras, which will prove the theorem. We accomplish this in the following three steps.

\begin{enumerate}
\item[{\bf Step 1.}] Give the definition of $P$;
\item[ {\bf Step 2.}] Verify that $P$ is a Rota-Baxter operator;
\item[{\bf Step 3.}] Show that $f$ is a Rota-Baxter algebra isomorphism.
\end{enumerate}

\noindent
{\bf Step 1. Definition of $P$.} For $k=0$, we define a $\bfk$-linear map $P: S^{-1}R=R\oplus V \to B$ by assigning $P(a)=P_R(a)$ and $P(v)=1\ot v$. For each $k\geq 1$, we define a $\bfk$-linear map by induction on $k$
$$P\,:\, S^{-1}R \ot  V^{\otimes k} =(R\ot V^{\otimes k})\oplus V^{\otimes k+1} \to B$$  by assigning
\begin{center}\begin{tabular}{rcl}
$P(\bar{v})$ & = & $P_{\sha S^{-1}R}(\bar{v})=1\otimes \bar{v}$ \quad where $\bar{v}\in V^{\otimes k+1}$\\
$P(a\otimes \bar{u})$ & = & $P(a)\otimes \bar{u}-P(P(a)\bar{ u})-\lambda P(a\bar{u}) $ \\
                      & = & $P(a)P(\bar{u})-P(P(a)\bar{u})-\lambda P(a\bar{u}) $ \quad where $a\in R $ and $\bar {u}\in V^{\otimes k} $.
\end{tabular}
\end{center}
One checks that $P$ extends by linearity to a well defined $\bfk$-linear map $P: B\to B$.
\smallskip

\noindent
{\bf Step 2. $P$ is a Rota-Baxter operator. }
Since $P$ is $\bfk$-linear, it is enough to show that
\begin{equation}\label{ToBeChecked} P(\alpha)P(\beta) - P\big(\alpha P(\beta)+P(\alpha) \beta +\lambda \alpha \beta\big)=0 \end{equation}
for any
$\alpha=(a+v)\otimes \overline{v}\in (R\oplus V)\otimes V^{\otimes n}, \beta=(b+u)\otimes \overline{u} \in (R\oplus V)\otimes V^{\otimes m}$ where
$$a, b\in R; \;  u, v\in V; \; \bar{v}= v_1\otimes\cdots\otimes v_n \in V^{\otimes n} \mbox{ and } \bar{u}= u_1\otimes\cdots\otimes u_m\in V^{\otimes m}.$$
Let $$\begin{array}{rcl}
I_1&=& P(a\otimes \overline{v})P(b\otimes \overline{u}) -P\big((a\otimes \overline{v}) P(b\otimes \overline{u})+P(a\otimes \overline{v}) (b\otimes \overline{u})+ \lambda (a\otimes \overline{v}) (b\otimes \overline{u}) \big)\\
I_2&=&  P(v\otimes \overline{v})P(b\otimes \overline{u})-P\big((v\otimes \overline{v}) P(b\otimes \overline{u})+P(v\otimes \overline{v}) (b\otimes \overline{u})+ \lambda (v\otimes \overline{v}) (b\otimes \overline{u}) \big) \\
I_3&=& P(a\otimes\overline{v})P(u\otimes\overline{u}) - P\big((a\otimes \overline{v}) P(u\otimes\overline{u})+P(a\otimes \overline{v})   (u\otimes \overline{u})+ \lambda (a\otimes \overline{v}) (u\otimes \overline{u}) \big)   \\
I_4&=& P(v\otimes \overline{v})P(u\otimes \overline{u})-P\big((v\otimes \overline{v}) P(u\otimes \overline{u})+P(v\otimes \overline{v})  (u\otimes \overline{u})+ \lambda (v\otimes \overline{v}) (u\otimes \overline{u}) \big)\\
\end{array}$$
Since $I_4=0$ because $P=P_{\sha S^{-1}A}$ when restricted to $V^{\otimes n}$ for any $n\geq 1$, we see that
 $$P(\alpha)P(\beta) - P\big(\alpha P(\beta)+P(\alpha) \beta +\lambda \alpha \beta\big)=I_1+I_2+I_3+I_4=I_1+I_2+I_3.$$

\begin{lemma}\label{mn=0} The identity $(\ref{ToBeChecked})$ holds when $mn=0$.
\end{lemma}
\begin{proof}
Assume that  $m=n=0$. In this case,  $I_1=0$ because $P=P_A$ when restricted to $A\subseteq B$. $I_2=I_3=0$ because of the definition of $P$. So the identity $(\ref{ToBeChecked})$ holds when $m=n=0$.

By symmetry we now assume that $m=0$ and prove the lemma by induction on $n$. The case when $n=0$ has been checked. In general, we first notice that $I_2=0$ because of the definition of $P$.
{\allowdisplaybreaks
\begin{eqnarray*}
 & & P(a\otimes \bar{v})P(b) - P\big( (a\otimes \bar{v})P(b) + P(a\otimes \bar{v})b + \lambda (a\otimes \bar{v}) b \big)\\
&\stackrel{(I)}{=} & \Big(P(a)P(\bar{v})-P\big(P(a)\bar{v}\big)-\lambda P(a\bar{v}) \Big)P(b) -P\Big( aP(\bar{v})P(b) + P\big(aP(\bar{v})\big)b
   + \lambda aP(\bar{v}) b\Big) \\
&= & P(a)P(\bar{v})P(b)-\lambda P(a\bar{v}) P(b) -P\big(P(a)\bar{v}\big)P(b) -P\Big( aP(\bar{v})P(b) +
    P\big(aP(\bar{v})\big)b+\lambda aP(\bar{v})b\Big)\\
&\stackrel{(II)}{=} &  P\big(aP(b)+P(a)b+\lambda ab \big)P(\bar{v})\\
&  & -\lambda P\big( a\bar{v}P(b)+ P(a\bar{v})b+\lambda a\bar{v}b\big)- P\Big(P(a)\bar{v}P(b)+P\big(P(a)\bar{v} \big) b
    +\lambda P(a)\bar{v}b\Big) \\
 & & - P\Big( aP(\bar{v})P(b) + P\big(aP(\bar{v})\big)b+\lambda aP(\bar{v})b \Big) \\
&\stackrel{(III)}{=}  &  P\Big(\big( aP(b)+P(a)b+\lambda ab\big) P(\bar{v})+
P\big(aP(b)+P(a)b+\lambda ab\big)\bar{v} + \lambda \big( aP(b)+P(a)b+\lambda ab \big)\bar{v}  \Big)\\
 & & -\lambda P\big( a\bar{v}P(b)+ P(a\bar{v})b+\lambda a\bar{v}b\big)- P\Big(P(a)\bar{v}P(b)+P\big(P(a)\bar{v} \big) b
    +\lambda P(a)\bar{v}b\Big) \\
 & & - P\Big( aP(\bar{v})P(b) + P\big(aP(\bar{v})\big)b+\lambda aP(\bar{v})b\Big) \\
&= &  P\Big(  aP(b)P(\bar{v})+P(a)bP(\bar{v})+\lambda abP(\bar{v})
             + P\big(aP(b)\big)\bar{v} + P\big(P(a)b\big)\bar{v}+ P\big(\lambda ab\big)\bar{v} \\
&  &\quad     + \lambda aP(b)\bar{v} +\lambda P(a)b \bar{v}+\lambda^2 ab \bar{v}  \Big) \\
&  & \!\!\!\! -P\Big( \lambda a\bar{v}P(b)+\lambda P(a\bar{v})b+\lambda^2 a\bar{v}b+P(a)\bar{v}P(b)+P\big(P(a)v \big) b
    +\lambda P(a)\bar{v}b \\
 & & \quad +aP(\bar{v})P(b) + P\big(aP(\bar{v})\big)b+\lambda aP(\bar{v})b \Big) \\
&= &  P\Big(  P(a)bP(\bar{v})
             + P\big(aP(b)\big)\bar{v} + P\big(P(a)b\big)\bar{v}+ P\big(\lambda ab\big)\bar{v}  \Big) \\
 & & \!\!\!\! -P\Big(\lambda P(a\bar{v})b+P(a)\bar{v}P(b)+P\big(P(a)\bar{v} \big) b  + P\big(aP(\bar{v})\big)b \Big) \\
&= & P\Big[ \Big( P(a)P(\bar{v}) - P\big( \lambda a\bar{v} + P(a)\bar{v} + aP(\bar{v})\big) \Big) b +
     \Big( P\big( aP(b)+  P(a)b + \lambda ab \big) -P(a)P(b)     \Big) \bar{v} \Big]    \\
&\stackrel{(IV)}{=} & P(b0+0\bar{v})=0.
\end{eqnarray*}
}

 In the above calculation, equation $(I)$ follows from the definition of $P$. Equations $(II)$ follows from the fact that $P=P_{A}$ on $A$ and the induction hypothesis because both $a\bar{v}$ and $P(a)\bar{v}$ are in  $S^{-1}A\otimes V^{\otimes {n-1}}$. Equations $(III)$ and $(IV)$ follow from the induction hypothesis because $\bar{v}\in  S^{-1}A\otimes V^{\otimes {n-1}}$. This shows that $I_1=0$.

 The proof that $I_3=0$ goes similarly where equation $(II)$ has to be modified.

{\allowdisplaybreaks
\begin{eqnarray*}
&  & P(a\otimes \bar{v})P(u) - P\big( (a\otimes \bar{v})P(u) + P(a\otimes \bar{v})u + \lambda (a\otimes \bar{v}) u \big)\\
&\stackrel{(I)}{=} & \Big(P(a)P(\bar{v})-P\big(P(a)\bar{v}\big)-\lambda P(a\bar{v}) \Big)P(u) -P\Big( aP(\bar{v})P(u) + P\big(aP(\bar{v})\big)u
   + \lambda aP(\bar{v}) u\Big) \\
&= & P(a)P(\bar{v})P(u)-\lambda P(a\bar{v}) P(u) -P\big(P(a)\bar{v}\big)P(u) -P\Big( aP(\bar{v})P(u) +
    P\big(aP(\bar{v})\big)u+\lambda aP(\bar{v})u\Big)\\
&\stackrel{(II)}{=} &  P(a)P\big( \bar{v}P(u)+P(\bar{v})u+\lambda \bar{v} u \big)\\
 & & -\lambda P\big( a\bar{v}P(u)+ P(a\bar{v})u+\lambda a\bar{v}u\big)- P\Big(P(a)\bar{v}P(u)+P\big(P(a)\bar{v} \big) u
    +\lambda P(a)\bar{v}u\Big) \\
  && - P\Big( aP(\bar{v})P(u) + P\big(aP(\bar{v})\big)u+\lambda aP(\bar{v})u \Big) \\
&\stackrel{(III)}{=}  &  P\Big( a P\big( \bar{v}P(u)+P(\bar{v})u+\lambda \bar{v} u \big) +
                         P(a) \big( \bar{v}P(u)+P(\bar{v})u+\lambda \bar{v} u  \big) +
                         \lambda a \big( \bar{v}P(u)+P(\bar{v})u+\lambda \bar{v} u \big)  \Big)\\
 & & -\lambda P\big( a\bar{v}P(u)+ P(a\bar{v})u+\lambda a\bar{v}u\big)- P\Big(P(a)\bar{v}P(u)+P\big(P(a)\bar{v} \big) u
    +\lambda P(a)\bar{v}u\Big) \\
&  & - P\Big( aP(\bar{v})P(u) + P\big(aP(\bar{v})\big)u+\lambda aP(\bar{v})u\Big) \\
&= &  P\Big( a P\big( \bar{v}P(u) \big)+a P\big(P(\bar{v})u\big) + \lambda a P\big(\bar{v} u \big) +
                         P(a)  \bar{v}P(u) + P(a) P(\bar{v})u  + \lambda P(a)\bar{v} u \\
&  &  \quad +\lambda a \bar{v}P(u)+\lambda a P(\bar{v})u+\lambda^2 a  \bar{v} u  \Big)\\
&  & \!\!\!\! -P\Big( \lambda a\bar{v}P(u)+\lambda P(a\bar{v})u+\lambda^2 a\bar{v}u+P(a)\bar{v}P(u)+P\big(P(a)v \big) u
    +\lambda P(a)\bar{v}u \\
&  & \quad +aP(\bar{v})P(u) + P\big(aP(\bar{v})\big)u+\lambda aP(\bar{v})u \Big) \\
&= &  P\Big( a P\big( \bar{v}P(u) \big)+a P\big(P(\bar{v})u\big) + \lambda a P\big(\bar{v} u \big) + P(a) P(\bar{v})u  \Big)\\
&  & \!\!\!\! -P\Big( \lambda P(a\bar{v})u+P\big(P(a)\bar{v} \big) u +aP(\bar{v})P(u) + P\big(aP(\bar{v})\big)u \Big) \\
&= & P\Big[ a \Big( P\big( \bar{v}P(u)+P(\bar{v})u + \lambda \bar{v}u \big) - P(\bar{v})P(u)  \Big)  +
     \Big( P(a)P(\bar{v})- P\big( \lambda a\bar{v}+  P(a)\bar{v} + aP(\bar{v}) \big) \Big) u \Big]    \\
&\stackrel{(IV)}{=} & P(a0+0u)=0.
\end{eqnarray*}
}
 In the above calculation, equation $(I)$ follows from the definition of $P$. Equation $(II)$ follows from the fact that $P=P_{\sha S^{-1}A}$ on $\bigoplus\limits_{n\geq 1}V^{\otimes n}$ and the induction hypothesis because both $a\bar{v}$ and $P(a)\bar{v}$ are in  $S^{-1}A\otimes V^{\otimes {n-1}}$. Equation $(III)$ follows from the definition of $P$ since $\bar{v}P(u)+P(\bar{v})u+\lambda \bar{v} u \in \bigoplus\limits_{n\geq 1}V^{\otimes n}$. Equation $(IV)$ follows from the induction hypothesis because $\bar{v}\in  S^{-1}A\otimes V^{\otimes {n-1}}$. This shows that $I_3=0$. The lemma is proved.
 \end{proof}

We now prove Eq.~(\ref{ToBeChecked}) by induction on $m+n$. Assume that that Eq.~(\ref{ToBeChecked}) holds when $m+n\leq k$ for some integer $k\geq 1$. In view of Lemma~\ref{mn=0}, we assume that $mn\neq 0$.  Recall that $\beta=(b+u)\otimes \bar{u}$. The following calculation shows that $I_1=0$ (let $u=0$ in $\beta$) and $I_3=0$ (let $b=0$ in $\beta$).
{\allowdisplaybreaks
\begin{eqnarray*}
&  & P(a\otimes \bar{v})P(\beta) - P\big( (a\otimes \bar{v})P(\beta) + P(a\otimes \bar{v})\beta + \lambda (a\otimes \bar{v}) \beta \big)\\
&\stackrel{(I)}{=} & \Big(P(a)P(\bar{v})-P\big(P(a)\bar{v}\big)-\lambda P(a\bar{v}) \Big)P(\beta) -P\Big( aP(\bar{v})P(\beta) + P\big(aP(\bar{v})\big)\beta
   + \lambda aP(\bar{v}) \beta\Big) \\
&= & P(a)P(\bar{v})P(\beta)-\lambda P(a\bar{v}) P(\beta) -P\big(P(a)\bar{v}\big)P(\beta) -P\Big( aP(\bar{v})P(\beta) +
    P\big(aP(\bar{v})\big)\beta+\lambda aP(\bar{v})\beta\Big)\\
&\stackrel{(II)}{=} &  P(a)P\big( \bar{v}P(\beta)+P(\bar{v})\beta+\lambda \bar{v} \beta \big)\\
&  & -\lambda P\big( a\bar{v}P(\beta)+ P(a\bar{v})\beta+\lambda a\bar{v}\beta\big)- P\Big(P(a)\bar{v}P(\beta)+P\big(P(a)\bar{v} \big) \beta
    +\lambda P(a)\bar{v}\beta\Big) \\
&  & - P\Big( aP(\bar{v})P(\beta) + P\big(aP(\bar{v})\big)\beta+\lambda aP(\bar{v})\beta \Big) \\
&\stackrel{(III)}{=}  &  P\Big( a P\big( \bar{v}P(\beta)+P(\bar{v})\beta+\lambda \bar{v} \beta \big) +
                         P(a) \big( \bar{v}P(\beta)+P(\bar{v})\beta+\lambda \bar{v} \beta  \big) +
                         \lambda a \big( \bar{v}P(\beta)+P(\bar{v})\beta+\lambda \bar{v} \beta \big)  \Big)\\
&  & -\lambda P\big( a\bar{v}P(\beta)+ P(a\bar{v})\beta+\lambda a\bar{v}\beta\big)- P\Big(P(a)\bar{v}P(\beta)+P\big(P(a)\bar{v} \big) \beta
    +\lambda P(a)\bar{v}\beta\Big) \\
&  & - P\Big( aP(\bar{v})P(\beta) + P\big(aP(\bar{v})\big)\beta+\lambda aP(\bar{v})\beta\Big) \\
&= &  P\Big( a P\big( \bar{v}P(\beta) \big)+a P\big(P(\bar{v})\beta\big) + \lambda a P\big(\bar{v} \beta \big) +
                         P(a)  \bar{v}P(\beta) + P(a) P(\bar{v})\beta  + \lambda P(a)\bar{v} \beta \\
&  &  \quad +\lambda a \bar{v}P(\beta)+\lambda a P(\bar{v})\beta+\lambda^2 a  \bar{v} \beta  \Big)\\
&  & \!\!\!\! -P\Big( \lambda a\bar{v}P(\beta)+\lambda P(a\bar{v})\beta+\lambda^2 a\bar{v}\beta+P(a)\bar{v}P(\beta)+P\big(P(a)v \big) \beta
    +\lambda P(a)\bar{v}\beta \\
&  & \quad +aP(\bar{v})P(\beta) + P\big(aP(\bar{v})\big)\beta+\lambda aP(\bar{v})\beta \Big) \\
&= &  P\Big( a P\big( \bar{v}P(\beta) \big)+a P\big(P(\bar{v})\beta\big) + \lambda a P\big(\bar{v} \beta \big) + P(a) P(\bar{v})\beta  \Big)\\
&  & \!\!\!\! -P\Big( \lambda P(a\bar{v})\beta+P\big(P(a)\bar{v} \big) \beta +aP(\bar{v})P(\beta) + P\big(aP(\bar{v})\big)\beta \Big) \\
&= & P\Big[ a \Big( P\big( \bar{v}P(\beta)+P(\bar{v})\beta + \lambda \bar{v}\beta \big) - P(\bar{v})P(\beta)  \Big)  +
     \Big( P(a)P(\bar{v})- P\big( \lambda a\bar{v}+  P(a)\bar{v} + aP(\bar{v}) \big) \Big) \beta \Big]    \\
&\stackrel{(IV)}{=} & P(a0+0\beta)=0.
\end{eqnarray*}
}
In the above calculation, equation $(I)$ follows from the definition of $P$. Equation $(II)$ follows from the induction hypothesis because $\bar{v}$, $a\bar{v}$ and $P(a)\bar{v}$ are in  $S^{-1}A\otimes V^{\otimes {(n-1)}}$. Equation $(III)$ follows from Lemma \ref{mn=0} and the fact that $P$ is $\bfk$-linear. Equation $(IV)$ follows from the induction hypothesis because $\bar{v}\in  S^{-1}A\otimes V^{\otimes {(n-1)}}$.  So $I_1=I_3=0$. Now it is easy to see that $I_2=0$  because it is symmetric to $I_3=0$.

Thus we have proved that $P$ is a Rota-Baxter operator and hence $(B, P_B)$ is a commutative Rota-Baxter algebra.
\smallskip

\noindent
{\bf Step 3. $f$ is a Rota-Baxter isomorphism.}
We first prove a lemma.
\begin{lemma} $B$ is generated by $S^{-1}A$ as a Rota-Baxter algebra. \mlabel{Lemma: Generators}
\end{lemma}
\begin{proof} Let $C$ be the Rota-Baxter subalgebra of $B$ generated by $S^{-1}A$. We need to show that for any $k\geq 0$, $S^{-1}A \otimes V^{k}$ is contained in $C$. We proof by induction on $k$. The case $k=0$ is trivial. For any $a\in S^{-1}A$ and $v\in V$, we have $P(v)\in C$. So $a\otimes v= aP(v)\in C$. This shows that $S^{-1}A\otimes V \subseteq C$. Now we assume that $S^{-1}A\otimes V^{k}\subseteq C$, then for any $a\in S^{-1}A$ and $\bar{v}\in V^{k+1}$, we have $\bar{v}\in S^{-1}A\otimes V^{k}\subseteq C$ and hence $P(\bar{v})\in C$. So $a\otimes \bar{v}= aP(\bar{v})\in C$. This proves that $S^{-1}A\otimes V^{k+1}\subseteq C.$
\end{proof}
It is easy to check that $f$ defined in Eq.~(\mref{eq:locmap}) is a morphism of Rota-Baxter algebras. There is another morphism of Rota-Baxter algebras
$$g: (A, P_A) \longrightarrow (B, P_B); \, a \mapsto a \in S^{-1}A \subseteq B, a\in A, $$
where we identify $a\in A$ with its image in $S^{-1}A$. Since the image of $S\subseteq A$ in $S^{-1}A$ and hence in $B$ is invertible, by the universal property of $(S^{-1}_{RB}A,S^{-1}P_A)$, there is a Rota-Baxter algebra  morphism $h: (S^{-1}_{RB}A, S^{-1}P_A) \to (B, P_B)$ which fits into the following commutative diagram.

$$
\xymatrix{
   (A, P_A) \ar[r]^-i \ar[dr]^-g  \ar[ddr]_i & (S^{-1}_{RB}A, S^{-1}P_A) \ar[d]^h \\
                                   & (B, P_B) \ar[d]^f \\
                                   & (S^{-1}_{RB}A, S^{-1}P_A)
}$$
By the universal property of localizations, $ f\circ h$ is the identity map of $(S^{-1}_{RB}A, S^{-1}P_A)$. So $h$ is injective. By the commutative diagram we see that the image of $h$ is a Rota-Baxter subalgebra of $B$ containing $S^{-1}A$.
By Lemma ~\mref{Lemma: Generators}, $(B, P_B)$ is generated by $S^{-1}A$ as a Rota-Baxter algebra. So the image of $h$ coincides with $B$, i.e., $h$ is also surjective. This proves that both $f$ and $h$ are isomorphisms and hence the theorem.            \end{proof}

\subsection{The weight zero case}

Let $(A, P)$ be a commutative Rota-Baxter algebra of weight $\lambda$ as before. Theorem \mref{thm:locnzero} gives concrete algebra structure of the Rota-Baxter localizations $S^{-1}_{RB}A$ under the condition that $S^{-1}A=A\oplus V$ where $V$ is closed under multiplication. If the weight $\lambda$ happens to be zero, the same result holds without requiring  that $V$ is closed under multiplication. Note that, if $\lambda=0$, the shuffle product on $S^{-1}R\oplus \bigoplus_{k\geq 1} (S^{-1}R \ot V^{\ot k})$ does not require any multiplicative structure on $V$.

\begin{theorem}
Let$A $ be a commutative $\bfk$-algebra and let $S$ be a multiplicative subset of $A$. Suppose $S^{-1}A=A \oplus V$ where $V$ is a $\bfk$-submodule.
Then
$$S^{-1}_\rb A = S^{-1}A\oplus \bigoplus_{k\geq 1} (S^{-1}A \ot V^{\ot k})$$
with the extended shuffle product. In other words, we have the tensor product algebra
$$ S^{-1}_\rb A = S^{-1}A \ot \Big(\bigoplus_{k\geq 0} V^{\ot k}\Big)
$$
where the second tensor factor is the shuffle product algebra on $V$.
\mlabel{thm:loczero}
\end{theorem}

As in the nonzero weight case, we note that the right hand side is a subalgebra, but not a Rota-Baxter subalgebra,  of $\sha(S^{-1}A)$.

\begin{proof}
Let $B=S^{-1}A\oplus \big( \bigoplus\limits_{k\geq 1} S^{-1}A \ot V^{\ot k}\big)$ be the algebra in the theorem.  We need to show that the composition map
\begin{equation}
f: B \hookrightarrow \sha(S^{-1}A) \to S^{-1}_\rb A
\mlabel{eq:comp}
\end{equation} is bijective. We proceed by first defining an Rota-Baxter operator $P=P_B$ on $B$ such that $B$ is generated by $S^{-1}A$ and there is a Rota-Baxter algebra morphism $g: (A, P_A)\to (B, P_B)$. By the universal property of localization,  $g$ induces a Rota-Baxter algebra morphism $h: (S^{-1}_{RB}A, S^{-1}P) \to (B, P_B)$. We then check that $h$ is the inverse of $f$.

We inductively define an operator $P=P_B: B \to B$ by assigning
\begin{center}\begin{tabular}{l}
$P(a)=P_A(a)$ \quad where $a\in A$\\
$P(\bar{v})= 1\otimes \bar{v}$ \quad where $\bar{v}\in V^{\otimes n}$\\
$P(a\otimes \bar{v}) = P_A(a)\otimes \bar{v}-P(P(a)\bar{ v})=P(a)P(\bar{v})-P(P(a)\bar{ v}) $ \quad where $a$ and $\bar {v}$ as above.
\end{tabular}
\end{center}
One checks easily that $P$ is the desired Rota-Baxter operator on $B$. The proof now goes exactly as the proof of Theorem ~\mref{thm:locnzero}.
\end{proof}

\begin{exam}
{\rm
Let $A=\bfk[x]$ be the ring of polynomials of one variable over a field $\bfk$ of character $0$. Let $\bfk(x)$ be the field of fractions of $A$.
Let $P$ be the integral operator which sends $x^n$ to $P(x^n):L=x^{n+1}/(n+1)$. So $P$ is a Rota-Baxter operator of weight zero. Let $S\subseteq A$ be the multiplicative subset $\{x^{n} \big\vert n\geq 1\}$. Clearly $S^{-1}A=\bfk[x, x^{-1}]=A \oplus V\subseteq \bfk(x)$ where $V$ is the $\bfk$-module $x^{-1}\bfk[x^{-1}]$. By Theorem ~\mref{thm:loczero}, the Rota-Baxter localization of $A$ at $S$ is given as the tensor algebra
$$\bfk[x,x^{-1}] \otimes \Big(\bigoplus\limits_{k\geq 0}(x^{-1}\bfk[x^{-1}])^{\otimes k} \Big), $$ where the multiplication in $\bigoplus\limits_{k\geq 0}(x^{-1}\bfk[x^{-1}])^{\otimes k}$ is given by the shuffle product $\sha$.
}
\end{exam}

\subsection{Localization of Rota-Baxter algebras with the zero Rota-Baxter operator}
Let $A$ be an $\bfk-$algebra. Then $A$ is naturally a Rota-Baxter algebra with the zero Rota-Baxter operator that sends elements of $A$ to the zero element of $A$. In fact, this defines a functor which embeds the category $\Comm/\bfk$ as a full faithful subcategory of  $\CRB/\bfk$. In this section, we study the Rota-Baxter localization of $A$ when $A$ has the zero Rota-Baxter operator. We assume that the weight of the operator is zero.

To simplify notations, we identify $a\in A$ with its image in the usual localization $S^{-1}A$.

\begin{theorem}
Let $A$ be a Rota-Baxter $\bfk$-algebra with the zero Rota-Baxter operator (of weight zero). Let $S$ be a multiplicative subset of $A$ and let $V$ be the $\bfk$-module $S^{-1}A/A$. Then the Rota-Baxter localization $S^{-1}_\rb A$ of $A$ at $S$ in $\CRB/\bfk$, with weight zero, is given by
$$ S^{-1}A\oplus \left(\bigoplus_{k\geq 1} (S^{-1}A \ot V^{\ot k})\right)$$
with the quotient ring structure from $\sha (S^{-1}A)$. In other words, we have the tensor product algebra
$$ S^{-1}_\rb A = S^{-1}A \ot \Big(\bigoplus_{k\geq 0} V^{\ot k}\Big)
$$
where the second tensor factor is the shuffle product algebra on $V$.
\mlabel{thm:loczeroweightzerooperation}
\end{theorem}

\begin{proof}
Consider the following subset of $\sha (S^{-1}A)$,
$$\begin{array}{ccl}J & = & \bigoplus\limits_{k\geq 1}^\infty\big(\sum\limits_{i=1}^n (S^{-1}A)^{\ot i}\ot A \ot (S^{-1}A)^{\ot (n-i)} \big) \\
                      & = & (S^{-1}A\ot A) \bigoplus  \left((S^{-1}A)^{\ot 2}\ot A + S^{-1}A \ot A \ot S^{-1}A\right) \bigoplus \cdots
\end{array}$$
Then we have
$$ S^{-1}A \oplus \left(\bigoplus_{k\geq 1} (S^{-1}A\ot V^{\ot k})\right) = \sha(S^{-1}A)/J.$$
On the other hand, let $I$ denote the Rota-Baxter ideal generated by $\{1\ot a | a\in A\}$. Since the Rota-Baxter operator is zero, by Theorem \mref{thm:exist}, we have $S^{-1}_\rb A = \sha (S^{-1}A) /I$. Thus to prove the theorem, we only need to prove the following lemma.

\begin{lemma} \begin{enumerate}
\item $J$ is a Rota-Baxter ideal of $\sha (S^{-1}A). $
\mlabel{it:Jideal}
\item
$I=J.$
\mlabel{it:equal}
\end{enumerate}
\end{lemma}
\begin{proof}
(\mref{it:Jideal}). $J$ is obviously a $\bfk$-module and closed under the Rota-Baxter operator of $\sha (S^{-1}A)$. It remains to check that
$J$ is a Rota-Baxter ideal of $\sha(S^{-1}A)$.
Note that $J$ is $\bfk$-linearly generated by pure tensors. Thus to prove that $\fraka \frakb \in J$ for any $\fraka\in J$ and $\frakb\in \sha(S^{-1}A)$, we only need to prove it for pure tensors $\fraka:=a_0\ot a_1 \ot \cdots \ot a_m\in J$ and $\frakb:=b_0\ot b_1\ot \cdots \ot b_n\in \sha(S^{-1}A) $, where $a_i, b_j$ are elements in $S^{-1}A$ and, for some $1\leq j_0\leq m$, $b_{j_0}$ is in $A$.
Since the weight of the Rota-Baxter algebra is zero, we have
$$ \fraka \diamond \frakb = a_0b_0\ot \big((a_1\ot \cdots \ot a_m) \ssha (b_1\ot \cdots \ot b_n)\big),$$
where $\ssha$ is the shuffle product. Since a shuffle of $a_1\ot \cdots \ot a_m$ and $b_1\ot \cdots \ot b_n$ comes from performing a special permutation on the tensor factors of $a_1\ot \cdots \ot a_m\ot b_1\ot \cdots \ot b_n$, one tensor factor in the shuffle is in $A$, implying that this shuffle is in $J$. Hence $\fraka \diamond \frakb$ is in $J$.

\smallskip

\noindent
(\mref{it:equal}).
Since $\{1\ot a | a\in A\}$ is contained in $J$, we see that $I$ is a subset of $J$.
It remains to show that $J\subseteq I$, for which it suffices to show that $ (S^{-1}A)^{\ot i}\ot A \ot (S^{-1}A)^{\ot (n-i)} \subseteq I$ for any $n\geq 1$ and $1\leq i\leq n $. We prove this by induction on $n$. If $n=1$, then $i=1$. For any $a\in S^{-1}R$ and $b\in R$, we have
$$1\ot b\in I \Rightarrow a\ot b=a \diamond (1\ot b) \in I.$$
So the statement holds for $n=1$. In general, consider $a_1 \ot \cdots \ot a_i \ot a \ot a_{i+1}\ot\cdots \ot a_n\in (S^{-1}A)^{\ot i}\ot A \ot (S^{-1}A)^{\ot (n-i)}$, where $a\in A$. If $i\geq 2, $ then we have
$$a_1 \ot \cdots \ot a_i \ot a \ot a_{i+1}\ot\cdots \ot a_n =a_1\diamond P(a_2 \ot \cdots \ot a_i \ot a \ot a_{i+1}\ot\cdots \ot a_n)\in I $$ because  $a_2 \ot \cdots \ot a_i \ot a \ot a_{i+1}\ot\cdots \ot a_n\in I$ by induction assumption. If $i=1$, we see that
$$I\ni (1\ot a)\diamond (a_1\ot \cdots \ot a_n)=a_1\ot a \ot a_2  \ot \cdots \ot a_n + \sum\limits_{i=2}^{n} a_1 \ot \cdots \ot a_i \ot a \ot a_{i+1}\ot\cdots \ot a_n. $$
We have just showed that $\sum\limits_{i=2}^{n} a_1 \ot \cdots \ot a_i \ot a \ot a_{i+1}\ot\cdots \ot a_n\in I$, so $a_1\ot a \ot a_2  \ot \cdots \ot a_n\in I$ as desired.
\end{proof}

This completes the proof of Theorem~\mref{thm:loczeroweightzerooperation}.
\end{proof}

\section{Rota-Baxter tensor product and Rota-Baxter Zariski topology}
\mlabel{sec:site}
In this section we study the tensor product of commutative Rota-Baxter algebras. We then extend the concepts of Zariski coverings and Zariski topology to the category of commutative Rota-Baxter algebras.

\subsection{Tensor products of Rota-Baxter algebras}
We construction the tensor product (coproduct) of two Rota-Baxter algebras in the category of Rota-Baxter algebras. Note that this is different from the tensor product in the category of algebras.

Let $\CRB/\bfk$ denote the category of commutative Rota-Baxter algebras of weight $\lambda$ over $\bfk$ as before.
\begin{defn} {\rm
Given a diagram of morphisms in $\CRB/\bfk$  $$(R_1, P_{R_1})\stackrel{f_1}{\longleftarrow}(R_0, P_{R_0}) \stackrel{f_2}{\longrightarrow}(R_2, P_{R_2}),$$ the colimit of the  diagram, if exists, is called the {\bf Rota-Baxter tensor product} of $(R_1,P_{R_1})$ and $(R_2,P_{R_2})$ over $(R_0,P_{R_0})$ and is denoted by $(R_1 \rbot_{R_0} R_2, P_{R_1} \ot_{P_{R_0}}  P_{R_2})$ or simply $R_1 \rbot_{R_0} R_2$.
In more concrete terms, let $f_i:(R_0,P_0)\to (R_i,P_i), i=1,2,$ be homomorphisms of commutative Rota-Baxter algebras. Their Rota-Baxter tensor product is a commutative Rota-Baxter algebra $(R_1 \rbot_{R_0} R_2, P_{R_1} \ot_{P_{R_0}}  P_{R_2})$ with Rota-Baxter algebra homomorphisms $k_i:(R_i,P_{R_i})\to (R_1 \rbot_{R_0} R_2, P_{R_1} \ot_{P_{R_0}}  P_{R_2}), i=1,2,$ such that, for any commutative Rota-Baxter algebra $(R,P_R)$ and Rota-Baxter algebra homomorphisms $\psi_i:(R_1,P_{R_1})\to (R,P_R), i=1,2,$ there is a unique morphism $$\psi_{1,2}:(R_1 \rbot_{R_0} R_2, P_{R_1} \ot_{P_{R_0}}  P_{R_2})\to (R,P_R)$$ such that $\psi_i=\psi_{1,2}\circ k_i, i=1,2$.
}
\end{defn}

Our construction of the Rota-Baxter tensor product will be based on free commutative Rota-Baxter algebras. We first give some properties of these algebras.

\begin{lemma}
\begin{enumerate}
\item
Let $(R,P_R)$ be a commutative Rota-Baxter algebra and let $\phi_R:\sha(R)\to R$ be the surjective Rota-Baxter algebra homomorphism  induced by the identity map  on $R$ and the freeness of $\sha(R)$. Let $j_R$ be the natural algebra morphism $R \to \sha (R)$. Then
\begin{equation}
\phi_R\circ j_R=\id_{R},
\mlabel{eq:id}
\end{equation}
and
$\ker \phi_R$ is the Rota-Baxter ideal of $\sha(R)$ generated by the set $\{P_R(r)-1\ot r\,|\, r\in R\}$.
\mlabel{it:rideal}
\item
For a homomorphism $f:R\to S$ of commutative algebras $R$ and $S$, let $\sha(f):\sha(R)\to \sha(S)$ denote the Rota-Baxter algebra homomorphism induced by the algebra homomorphism $j_S\circ f:R\to \sha(S)$. Let Let $j_S$ be the natural map $S \to \sha (S)$. Then $j_S\circ f=\sha(f)\circ j_R$ and, for $r_0\ot \cdots \ot r_k\in \sha(R)$, we have
$$ \sha(f)(r_0\ot \cdots \ot r_k)=f(r_0)\ot \cdots \ot f(r_k).$$
\mlabel{it:hom}
\end{enumerate}
\mlabel{lem:pref}
\end{lemma}
\begin{proof}
(\mref{it:rideal}). Eq.~(\mref{eq:id}) follows from the freeness of $\sha(R)$. We next prove the statement on $\ker\phi_R$. By Eq.~(\mref{eq:id}), we have the splitting
$\sha(R)=R\oplus \ker \phi_R,$ where $R$ is identified with the image of $j_R$. Let $\cali_R$ denote the ideal of $\sha(R)$ generated by the set $\{P_R(r)-1\ot r\,|\,r\in R\}$.
Since $\ker \phi_R$ contains this set, we have $\ker\phi_R\supseteq
\cali_R$. We just need to prove
$$\sha(R)=R+ \cali_R.$$
For this, we just need to show that, any pure tensor $\frakr:=r_0\ot\cdots \ot r_k\in \sha(R)$ with $r_0,\cdots, r_k\in R$ is in $R+ \cali_R$. If $k=0$, then $\frakr$ is in $R$ and we are done. If $k\geq 1$, then we have
$$ \frakr=r_0\ot r_1\ot \cdots \ot r_k=r_0P_R(r_1P_R(\cdots P_R(r_k)\cdots)) -\big(r_0P_R(r_1P_R\cdots P_R(r_k)\cdots)-r_0\ot r_1\ot \cdots \ot r_k\big).$$
The first term is in $R$. We next use induction on $k\geq 1$ to prove the claim that the second term on the right hand side of the above equation is in $\cali_R$.

The claim holds when $k=1$ since
$r_0P_R(r_1)-r_0\ot r_1 = r_0(P_R(r_1)-1\ot r_1)$ is clearly in $\cali_R$. Suppose the claim has been proved for $k$ and consider the case of $k+1$. Then
\begin{eqnarray*}
&&r_0P_R(r_1P_R(\cdots P_R(r_{k+1})\cdots ))-r_0\ot r_1\ot\cdots \ot r_{k+1}
\\
&=&(r_0P_R(r_1P_R(\cdots P_R(r_{k+1})\cdots )) - r_0\ot (r_1P_R(\cdots P_R(r_{k+1})\cdots)))
\\&&+ r_0\ot\big( r_1P_R(\cdots P_R(r_{k+1})\cdots)-r_1\ot \cdots \ot r_{k+1}\big).
\end{eqnarray*}
On the right hand of the equation, the term in the first line is in $\cali_R$ by the case of $k=1$ and the term in the second line is in $\cali_R$ by the induction hypothesis. Hence the induction is completed.
\smallskip

\noindent
(\mref{it:hom}).
Note that in any free Rota-Baxter algebra $(\sha(A), P)$, by the definition of multiplication and $P$, we have the following important property
$$P(a_0\ot a_1\ot \cdots \ot a_n)=a_0P(a_1P(\cdots P(a_n)\cdots), \quad a_i\in A.$$
Then (\mref{it:hom}) follows since
\begin{eqnarray*}
\sha(f)(r_0\ot \cdots \ot r_k)&=&\sha(f)(r_0P_R(\cdots P_R(r_k)\cdots)) \\
&=&\sha(f)(r_0)P_S(\cdots P_S(\sha(f)(r_k) \cdots )) \\
&=& f(r_0)\ot \cdots \ot f(r_k).
\end{eqnarray*}
\end{proof}

We now establish the existence of the Rota-Baxter tensor product.
For a given diagram of morphisms in $\CRB/\bfk$
$$(R_1, P_{R_1})\stackrel{f_1}{\longleftarrow}(R_0, P_{R_0}) \stackrel{f_2}{\longrightarrow}(R_2, P_2),$$
let $R_1\ot_{R_0} R_2$ be the tensor product of $R_1$ and $R_2$ as $R_0$-algebras. Let $(\sha (R_1\ot_{R_0}R_2), P_{R_1\ot_{R_0}R_2} )$ be the free commutative Rota-Baxter algebra on $R_1\ot_{R_0}R_2$. For distinction, we use $ \Box $ to denote the tensor product in $R_1\ot_{R_0}R_2$ and use the usual $\otimes$ to denote the tensor product in $\sha (R_1\ot_{R_0}R_2)$.

\begin{theorem}
Let $\cali$ be the Rota-Baxter ideal of $\sha(R_1\ot_{R_0}R_2)$ generated by the set
$$ \big\{ P_{R_1}(r_1)\Box 1 - (1\Box1) \otimes (r_1\Box 1) \big\vert  r_1\in R_1 \big\}\cup
 \big\{1\Box P_{R_2}(r_2) - (1\Box1) \otimes (1\Box r_2) \big\vert  r_2\in R_2   \big\}.$$
Then the quotient Rota-Baxter algebra $\big( \sha (R_1\ot_{R_0}R_2)/\cali, \widetilde{P_{R_1\ot_{R_0}R_1}}\big)$
is the Rota-Baxter tensor product of $(R_1, P_{R_1})$ and $(R_2, P_{R_2})$ over $(R_0, P_{R_0})$.
\mlabel{thm:ten}
\end{theorem}
\begin{proof}
We use the following diagram to organize the notations that will be introduced in the proof. To simplify the notations, for $i=0,1,2$, we let $P_i$ denote the Rota-Baxter operator $P_{R_i}:\sha(R_i)\to \sha(R_i)$ and let $j_i:R_i\to \sha(R_i)$ denote the natural embedding. Also denote $j_{1,2}:=j_{R_1\ot_{R_0}R_2}:R_1\ot_{R_0}R_2 \to \sha(R_1\ot_{R_0}R_2)$ for the natural embedding.

$$
{\small
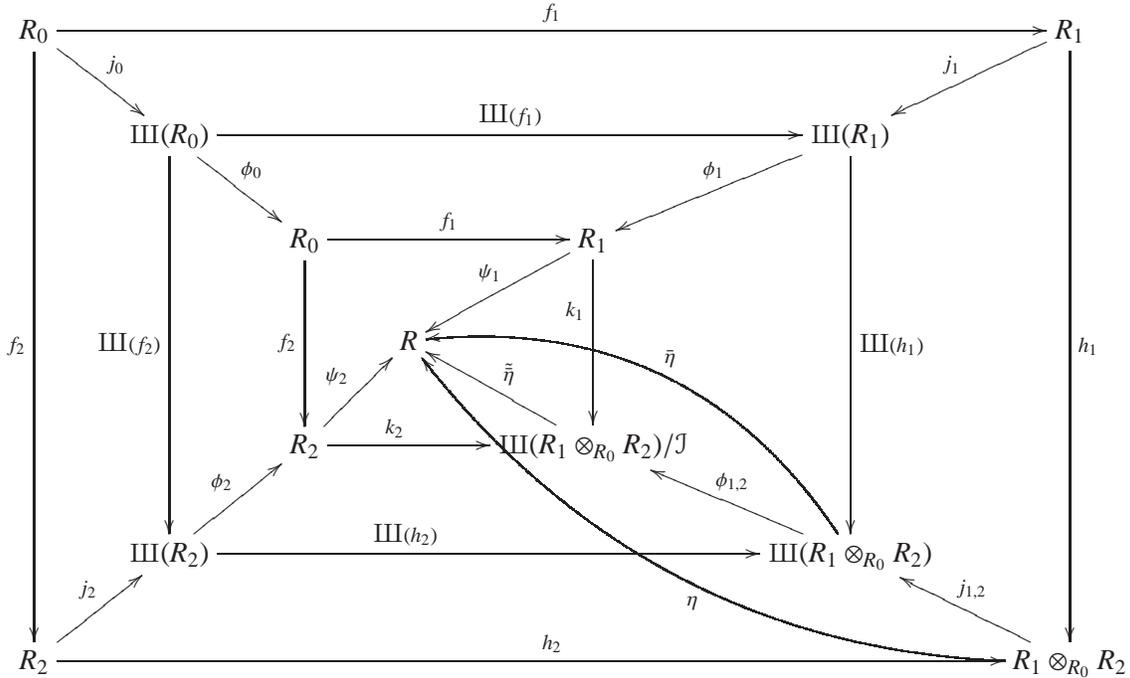
\xymatrix{
R_0\ar^{j_{0}}[dr] \ar^{f_1}[rrrrrr]
    \ar_{f_2}[dddddd] &&&&&& R_1 \ar_{j_{1}}[dl]
    \ar^{h_1}[dddddd]\\
&     \sha(R_0) \ar^{\sha(f_1)}[rrrr]\ar^{\phi_0}[rd] \ar_{\sha(f_2)}[dddd] &&&& \sha(R_1)\ar_{\phi_1}[ld] \ar^{\sha(h_1)}[dddd] &\\
&    & R_0\ar^{f_1}[rr]\ar_{f_2}[dd] &&
            R_1 \ar_{\psi_1}[ld]\ar@/-1pc/_{k_1}[dd]&\\
   & && R &&& \\
 &   & R_2\ar^{\psi_2}[ru] \ar@/-1.7pc/[rr]^{k_2} && \sha(R_1\ot_{R_0}R_2)/\cali \ar_{\tilde{\free{\eta}}}[lu]&&\\
&     \sha(R_2)\ar^{\phi_2}[ru] \ar@/-3.3pc/^{\sha(h_2)}[rrrr]&&&& {\scriptsize \sha(R_1\ot_{R_0}R_2)}
     \ar_{\phi_{1,2}}[lu] \ar@/_2.5pc/_{\free{\eta}}[uull] & \\
    R_2\ar^{j_{2}}[ur] \ar^{h_2}[rrrrrr] &&&&&& R_1\ot_{R_0}R_2 \ar_{j_{1,2}}[ul] \ar@/^3pc/^{\eta}[uuulll]
}
}
$$

Let $h_i:R_i\to R_1\ot_{R_0}R_2, i=1,2,$ be the algebra homomorphisms to the tensor product algebra $R_1\ot_{R_0}R_2$. By \cite[Proposition 3.4]{Gudom}, $(\sha(R_1\ot_{R_0}R_2),P_{R_1\ot_{R_0}R_2})$ is the tensor product of the Rota-Baxter $(\sha(R_0),P_0)$-algebras $(\sha(R_1),P_{R_1})$ and $(\sha(R_2),P_{R_2})$ in the category of commutative Rota-Baxter algebras.
By Lemma~\mref{lem:pref}.(\mref{it:hom}), all the outer trapezoids in the above diagram are commutative:
\begin{equation}
 j_i\circ f_i=\sha(f_i) \circ j_0, \quad
    j_{1,2}\circ h_i = \sha(h_i)\circ j_i, \quad i=1, 2.
 \mlabel{eq:outer}
\end{equation}


Let $\cali'$ be the Rota-Baxter ideal of $\sha(R_1\ot_{R_0}R_2)$ generated by
$\sha(h_1)(\ker \phi_1) \cup \sha(h_2)(\ker \phi_2).$
\begin{lemma}
$\cali'$ equals to the Rota-Baxter ideal $\cali$ defined in Theorem~\mref{thm:ten}.
\mlabel{lem:ideal}
\end{lemma}
\begin{proof}
Since $\{P_i(r_i)-1\ot r_i\,|\,r_i\in R_i\}$ is the generating set of the ideal $\cali_i$ of $R_i$ for $i=1, 2$, we find that $\sha(h_i)(\{P_i(r_i)-1\ot r_i\,|\,r_i\in R_i\})$ is the generating set of the ideal of $\sha(R_1\ot_{R_0}R_2)$ generated by $\sha(h_i)(\ker\phi_i).$ Since $h_1(r_1)=r_1\Box 1$ with the notation $\Box$ introduced before Theorem~\mref{thm:ten}, by Lemma~\mref{lem:pref}.(\mref{it:hom}) we have
$$\sha(h_1)(P_1(r_1)-1\ot r_1)=h_1(P_1(r_1)) - h_1(1)\ot h_1(r_1)=P_1(r_1)\Box 1 - (1\Box 1)\ot (r_1\Box 1)$$
and similarly
$$\sha(h_2)(P_2(r_2)-1\ot r_2)=1\Box P_2(r_2)-(1\Box 1)\ot (1\Box r_2).$$
Hence the lemma follows.
\end{proof}

Let $\phi_{1,2}:\sha(R_1\ot_{R_0}R_2)\to
\sha(R_1\ot_{R_0}R_2)/\cali$ be the quotient Rota-Baxter algebra homomorphism.
Then by the definition of $\cali=\cali'$, there are unique
$k_i:(R_i,P_i)\to (\sha(R_1\ot_{R_0}R_2)/\cali, \widetilde{P_{R_1\ot_{R_0}R_2}}), i=1,2,$
such that
\begin{equation}
k_i\circ \phi_i=\phi_{1,2}\circ \sha(h_i),\quad i=1, 2.
\mlabel{eq:kf}
\end{equation}

We now show that $(\sha(R_1\ot_{R_0}R_2)/\cali, \widetilde{P_{R_1\ot_{R_0}R_2}})$, together with the homomorphisms of Rota-Baxter algebras $k_i:R_i\to \sha(R_1\ot_{R_0}R_2), i=1, 2$, is the tensor product Rota-Baxter algebra of $(R_1,P_1)$ and $(R_2,P_2)$ over $(R_0,P_0)$. We achieve this by proving the following two lemmas. \begin{lemma}
$k_1\circ f_1=k_2\circ f_2. $
\mlabel{lem:kf}
\end{lemma}
\begin{proof}
By the surjectivity of $\phi_0$, we only need to prove
\begin{equation}
 k_1\circ f_1\circ \phi_0 = k_2\circ f_2 \circ \phi_0.
\mlabel{eq:kff}
\end{equation}

By the functorality of the functor $\sha$ from the category of commutative algebras to the category of commutative Rota-Baxter algebras, we obtain
\begin{equation}
 \sha(h_1)\circ \sha(f_1)=\sha(h_2)\circ \sha(f_2).
\mlabel{eq:shfh}
\end{equation}
Further, by Eq.~(\mref{eq:id}) we trivially have
$$f_i \circ \phi_0\circ j_0 = \phi_i\circ j_i \circ f_i,
\quad i=1,2.$$
By Eq.~(\mref{eq:outer}), we have
$$ f_i \circ \phi_0\circ j_0 = \phi_i\circ \sha(f_i) \circ j_0,
\quad i=1,2.$$
Thus by the freeness of $\sha(R_0)$, we have
\begin{equation}
f_i \circ \phi_0 = \phi_i\circ \sha(f_i), \quad i=1,2.
\mlabel{eq:ff}
\end{equation}
Combining equations (\mref{eq:kf}), (\mref{eq:shfh}) and (\mref{eq:ff}) we obtain
\begin{eqnarray*}
k_1\circ f_1 \circ \phi_0 &=& k_1\circ \phi_1 \circ \sha(f_1)\\
&=& \phi_{1,2}\circ \sha(h_1)\circ \sha(f_1)\\
&=& \phi_{1,2}\circ \sha(h_2)\circ \sha(f_2)\\
&=& k_2\circ \phi_2\circ \sha(f_2)\\
&=& k_2\circ f_2\circ \phi_0,
\end{eqnarray*}
as needed.
\end{proof}

\begin{lemma}
Let $(R,P)$ be any $(R_0,P_0)$-algebra and let $\psi_i: (R_i,P_i)\to (R,P), i=1,2$ be such that
$\psi_1\circ f_1=\psi_2\circ f_2$. There is unique Rota-Baxter homomorphism $\psi_{1,2}:\sha(R_1\ot_{R_0}R_2)/\cali \to R$ such that \begin{equation}
\psi_i= \psi_{1,2}\circ k_i,\quad i=1,2.
\mlabel{eq:pek}
\end{equation}
\mlabel{lem:pek}
\end{lemma}
\begin{proof}
From $\psi_1\circ f_1=\psi_2\circ f_2$ and Eq.~(\mref{eq:id}), we have

$$\psi_1\circ \phi_1\circ j_1\circ f_1=\psi_2\circ \phi_2\circ j_2 \circ f_2.$$
Thus by the universal property of the tensor product $R_1\ot_{R_0}R_2$, there is unique $\eta:R_1\ot_{R_0}R_2 \to R$ such that
\begin{equation}
\psi_i\circ \phi_i\circ j_{R_i}= \eta\circ h_i, i=1,2.
\mlabel{eq:ppj}
\end{equation}

By the universal property of the free Rota-Baxter algebra $(\sha(R_1\ot_{R_0}R_2),P_{R_1\ot_{R_0}R_2})$, there is unique
$\free{\eta}:(\sha(R_1\ot_{R_0}R_2),P_{R_1\ot_{R_0}R_2})
\to (R,P)$ such that
$\free{\eta}\circ j_{1,2}=\eta$. Combining with Eq.~(\mref{eq:ppj}) we get
$$\psi_i\circ \phi_i\circ j_i=\free{\eta}\circ j_{1,2}\circ h_i
=\free{\eta}\circ \sha(h_i)\circ j_i, \quad i=1,2.$$
Thus
$$\psi_i\circ \phi_i= \free{\eta}\circ \sha(h_i),\quad  i=1,2,$$
by the freeness of $\sha(R_1)$ and $\sha(R_2)$.

Further,
$$\free{\eta}(\sha(h_i)(\ker \phi_i))=(\free{\eta}\circ \sha(h_i))(\ker \phi_i)=(\psi_i\circ \phi_i)(\ker \phi_i)=0, \quad i=1,2.$$
Thus by the construction of $\cali$, there is unique $\tilde{\free{\eta}}: (\sha(R_1\ot_{R_0}R_2)/\cali, \widetilde{P_{R_1\ot_{R_0}R_2}}) \to (R,P)$ such that
$\free{\eta}=\tilde{\free{\eta}}\circ \phi_{1,2}.$
Hence
$$\psi_i\circ \phi_i=\free{\eta}\circ \sha(h_i)=\tilde{\free{\eta}}\circ \phi_{1,2}\circ \sha(h_i) =\tilde{\free{\eta}}\circ k_i\circ \phi_i, \quad i=1,2.$$
Then by the surjectivity of $\phi_i$, we get

$$\psi_i= \tilde{\free{\eta}}\circ k_i,\quad i=1,2.$$
So we just need to take $\psi_{1,2}=\tilde{\free{\eta}}$ for the existence of $\psi_{1,2}$.

To prove the uniqueness of $\psi_{1,2}$, suppose $\psi_{1,2}':(\sha(R_1\ot_{R_0}R_2)/\cali, \widetilde{P_{R_1\ot_{R_0}R_2}}) \to (R,P)$ such that
$\psi_i= \psi_{1,2}' \circ k_i,\quad i=1,2.$. Then tracing the above argument back, we obtain
$$ \psi=\psi_i\circ \phi_i\circ j_i=\psi_{1,2}'\circ \phi_{1,2}\circ j_{1,2}\circ h_i, \quad i=1,2.$$
By the uniqueness in the universal property of the tensor product $R_1\ot_{R_0}R_2$ of commutative algebras, we get
$$\psi_{1,2}'\circ \phi_{1,2}\circ j_{1,2} =\eta.$$
Then by the freeness of $\sha(A_1\ot_{A_0}A_2)$, we obtain
$$ \psi_{1,2}'\circ \phi_{1,2} = \free{\eta}.$$
Since we also have $\free{\eta}=\tilde{\free{\eta}}\circ \phi_{1,2}$, by the surjectivity of $\phi_{1,2}$, we obtain $\phi_{1,2}'\circ \phi_{1,2}=\tilde{\free{\eta}},$ proving the uniqueness of $\psi_{1,2}$.
\end{proof}

The proof of Theorem~\mref{thm:ten} is obtained by combining Lemma~\mref{lem:ideal}, \mref{lem:kf} and \mref{lem:pek}.
\end{proof}

\subsection{Rota-Baxter Zariski topology}
We end this paper by discussing a possible way to define the Zariski topology on the category of commutative Rota-Baxter algebras. The Zariski topology we want to define is a {\it Grothendieck topology} for which we now briefly recall the definition. Let $\mathfrak{C}$ be a category.   A Grothendieck topology on $\mathfrak{C}$ is a system of distinguished families of maps $\{U_i \to U\}_{i\in I}$ in $\mathfrak{C}$, called coverings,  such that
\begin{enumerate} \item[(a)] for any coverings $\{U_i \to U\}_{i\in I}$ and any morphism $V\to U$ in $\mathfrak{C}$, the fiber product  $U_i \underset{U}{\times}V$ exists for any $i$ and $\{U_i \underset{U}{\times}V \}_{i\in I}\to V$ is a covering of $V$;
\item[(b)] if $\{U_i \to U\}_{i\in I}$ is a covering of $U$ and, for each $i\in I$, $\{V_{ij}\to U_i\}_{j\in J_i}$ is a covering of $U_i$, then $\{V_{ij}\to U\}_{j\in J_i, i\in I}$ is covering, where $V_{ij}\to U$  is the obvious composition;
\item[(c)] for any $U\in \mathfrak{C}$, the family $\{U\stackrel{=}{\to}U\}$ consisting only the identity map is a covering.
\end{enumerate}
A {\it site} is a category $\mathfrak{C}$ together with a Grothendieck topology defined on $\mathfrak{C}$.

Next, we define the Zariski open coverings in $\CRB/\bfk$.
\begin{defn} {\rm
A morphism $f:(A, P_A)\to (B, P_B)$ in $\CRB/\bfk$ is called principle open if there exists a multiplicative subset $S\subseteq A$ and an isomorphism  $\tau $ such that there is a commutative diagram in $\CRB/\bfk$
$$\xymatrix{(A, P_A) \ar[r]^f \ar[dr]^{i} & (B, P_B) \ar[d]^{\tau}_{\cong}\\
 &(S^{-1}_{RB}A, S^{-1}P_A)
}$$ where $i$ is the natural map from $(A, P_A)$ to its localization at $S$. A morphism $g:(A, P_A)\to (B, P_B)$ is called open if $g$ is the composition of finitely many principle open morphisms. A family of maps $\{(A, P_A) \to (B_i, P_{B_i})\}_{i\in I}$ is called a Rota-Baxter Zariski covering of $(A, P_A)$ if there exists a multiplicative subset $S\in A$ such that
\begin{enumerate}\item[(1)] the Rota-Baxter ideal generated by $S$ is $A$;
\item[(2)] for any $s\in S$, there exists $i\in I$ such that the image of $s$ under $(A, P_A)\to (B_i, P_{B_i})$ is invertible.
\end{enumerate}
}
\end{defn}

Before stating the main result, we give the compatibility between localization and tensor product for Rota-Baxter algebras.

\begin{lemma} \label{TensorAndLocalization} Let $f:(A, P_A) \to (B, P_B)$ be a morphism in $\CRB/\bfk$ and $S$ be a multiplicative subset in $A$. Let $T$ be the multiplicative subset $f(S)$. Then there is a natural isomorphism
$$(S^{-1}_{RB}A \rbot_A B, S^{-1}P_A \ot_{P_A} P_B) \stackrel{\cong}{\longrightarrow} (T^{-1}_{RB}B, T^{-1}P_B).$$
\end{lemma}
\begin{proof} The proof is a routine exercise in category theory. We start by explaining that there is a commutative diagram in $\CRB/\bfk$
$$\xymatrix{
(A, P_A) \ar[r]^-{i_A} \ar[d]^-{f}   & (S^{-1}_{RB}A, S^{-1}P_A) \ar[d]^{j_2} \ar@/^1pc/[ddr]^{k} \\
(B, P_B) \ar[r]^-{j_1} \ar@/_1pc/[drr]_{i_B}                 & (S^{-1}_{RB}A \rbot_{A} B, S^{-1}P_A \ot_{P_A}  P_B)  \ar@{-->}[dr]^-g \\
                                       &       & (T^{-1}_{RB}B, T^{-1}P_B)  \ar@<1ex>@{-->}[ul]^-h  .
}$$
Morphisms  $i_A$ and $i_B$ are the  natural map into their localizations. The upper left square is the push-out square defining the tensor product $(S^{-1}_{RB}A \rbot_{A} B, S^{-1}P_A \ot_{P_A}  P_B)$. $k$ is induced from the map $i_B\circ f$. $g$ is the map induced by $k$ and $i_B$ since $k\circ i_A=i_B\circ f$. Since $j_1(T)=j_1(f(S)) =j_2(i_A(S))$ and $i_A(S)$ is a set of invertible elements, $j_1(T)$ is a set of invertible elements. So $j_1$ induces a morphism $h$.

Since $g\circ h\circ i_B = g\circ j_1 =i_B$,  the map $g\circ h$ is the identity map by the universal property of localizations. It is easy to see that $h\circ g \circ j_1 = h\circ i_B= j_1$.  We also have $$h\circ g \circ j_2\circ i_A = h\circ g\circ j_1\circ f = h\circ i_B \circ f =j_1\circ f=j_2\circ i_A,$$
which implies by the universal property of localizations that $h\circ g \circ j_2=j_2 $. So $h\circ g$ is the identity map by the universal property of tensor products. So $g$ and $h$ are isomorphisms as desired.
\end{proof}

\begin{theorem} The collection of Rota-Baxter Zariski coverings forms a Grothendieck topology on $(\CRB/\bfk)^{op}$, the opposite category of $\CRB/\bfk$.
\end{theorem}
\begin{proof} We check that this collection of coverings satisfies $(a), (b)$ and $(c)$ in the definition of Grothendieck topologies. Requirement $(c)$ is trivial because $(A, P_A)$ is the localization of itself at $1$. Let $f: (A, P_A)\to (B, P_B)$ and $g:(B, P_B)\to (C, P_C)$ be open morphisms. It follows from the definition that $g\circ f$ is again open. If $s\in A$ is an element such that $f(s)$ is invertible in $B$, then $g(f(s))$ is invertible in $C$. This shows that requirement $(b)$ is satisfied. Fiber product in $(\CRB/\bfk)^{op}$ is the tensor product in $\CRB/\bfk$ which always exists by Theorem~\mref{thm:ten}. To check requirement $(a)$, it suffices to check that principle open morphisms are closed under cobase changes in $\CRB/\bfk$, which is true by Lemma \ref{TensorAndLocalization}.
\end{proof}

The Grothendieck topology defined by the  Rota-Baxter Zariski coverings on $(\CRB/\bfk)^{op}$ is called the {\bf Rota-Baxter Zariski topology} and we denote the corresponding site by $RBZar/\bfk$.
%
A sheaf theory on this site is currently unavailable. It would be interesting to check whether the category of sheaves on this site
is an abelian category with enough injective objects. This property would enable us to discuss
Rota-Baxter sheave cohomology on commutative Rota-Baxter algebras.




\end{document}